\documentclass{amsart} 
\usepackage{amsmath, amssymb, verbatim, amscd,amsthm,mathrsfs,mathtools}
\usepackage[normalem]{ulem}
\usepackage[driverfallback=dvipdfm]{hyperref}
\usepackage[dvipdfm]{graphicx}
\usepackage{color,graphicx,shortvrb}
\usepackage{enumerate}

\usepackage[all]{xy}

\newtheorem{thm}{Theorem}[section]

\newtheorem{lem}[thm]{Lemma}
 
\newtheorem{cor}[thm]{Corollary} 
\newtheorem*{problem}{Problem}
\newtheorem*{oproblem}{Open Problem}

\theoremstyle{definition}

\newtheorem{rem}{Remark}
\newtheorem*{nota}{Notation} 
\newtheorem*{ack}{Acknowledgement} 
\numberwithin{equation}{section}

\newtheorem*{question*}{Question}

 \newcommand{\N}{\mathbb N}

 \newcommand{\V}[1]{\left\Vert #1 \right\Vert}

 \newcommand{\wT}{\widehat{T}}

 \newcommand{\set}[1]{\left\{ #1 \right\}}

 \newcommand{\wh}[1]{\widehat{#1}}

  \newcommand{\sn}[1]{\Vert #1 \Vert}

 \newcommand{\ko}{C^{+}(K_1)^{{-1}}}
 \newcommand{\kt}{C^{+}(K_2)^{{-1}}}

\newcommand{\pp}[1]{P_{L_2}(#1)}

 \newcommand{\pcx}{C_0^+(L_1)}
 \newcommand{\pcy}{C_0^+(L_2)}
 \newcommand{\pcz}{C_0^+(L_i)}

 \newcommand{\Vinf}[1]{\V{#1}}

  \newcommand{\zlb}{\mathbf{0}_{L_2}}

\usepackage{ulem}
\usepackage{marginnote}
\usepackage{geometry}

\usepackage{ulem}
\usepackage{marginnote}
\usepackage{geometry}

\begin{document}

\title[Norm additive mappings between commutative $C^{*}$-algebras in the range]
{The Cauchy equation and norm additive mappings between positive cones of commutative $C^{*}$-algebras}

\author[D.~Hirota]{Daisuke Hirota}
\dedicatory{%
This is a preprint of a paper submitted to the 
\textit{Journal of Mathematical Analysis and Applications (JMAA)}.
The paper is currently under minor revision.%
}

\address[Yamagata 997-8511 Japan]{National Institute of Technology, Tsuruoka College}
\email{dhirota@tsuruoka-nct.ac.jp}


\begin{abstract} 
Let \( A_i \) be a commutative \( C^{*} \)-algebra for \( i = 1, 2 \), and denote by \( A_i^{+} \) its positive cone, consisting of all positive elements of \( A_i \). In this paper, 
we investigate surjective, not necessarily continuous mappings \( T: A_1^{+} \to A_2^{+} \) that satisfy the norm equality
\[
\| T(a + b) \| = \| T(a) + T(b) \| \quad (a, b \in A_1^{+}).
\]
We prove that such a mapping \( T \) is necessarily additive and positive homogeneous. 
Furthermore, we show that if the mapping $T:A_{1}^{+}\to A_{2}^{+}$ between the positive cones of two unital commutative $C^{*}$-algebras $A_{i}$
with the unit element \( 1_{A_i} \) for \( i = 1, 2 \),  
and if \( T \) is also injective, 
then $T(1_{A_1})^{-1}T$ is a composition operator.  

This is the submitted version of a paper currently under minor revision for the Journal of Mathematical Analysis and Applications.
\end{abstract}

\maketitle

\section{Introduction and main theorem}
The Cauchy equation
\begin{equation}\label{Ca}
f(x+y)=f(x)+f(y)
\end{equation}
is one of the most classical functional equations and plays a fundamental role in the theory of functional equations.
A function $f$ satisfying \eqref{Ca} is called additive, and such functions are central in functional analysis.
For real-valued functions defined on the real line $\mathbb{R}$, it is known that every Lebesgue measurable solution of \eqref{Ca} is linear, i.e., of the form $f(x) = cx$ for some constant $c \in \mathbb{R}$.
For further details on this classical result, see \cite{AD}.

The motivation for generalizing the Cauchy equation \eqref{Ca} can be traced back to the works of M.~Hossz\'{u} and E.~Vincze.
Hossz\'{u} introduced a functional equation involving squares:
\[
f(x+y)^2 = (f(x) + f(y))^2,
\]
where $f$ is a real-valued function defined on $\mathbb{R}$, and showed in \cite{MH} that this equation is equivalent to the Cauchy equation \eqref{Ca} under mild regularity assumptions.
Building on this idea, Vincze considered a further generalization involving arbitrary natural exponents.
For any natural number $n$, he demonstrated that the equation
\[
f(x+y)^n = (f(x) + f(y))^n
\]
is also equivalent to \eqref{Ca}, provided that $x, y$ belong to a commutative semigroup $(Q, +)$ and $f$ is a complex-valued function on $Q$; see \cite{EV1, EV2, EV3}.

These early developments laid the groundwork for norm-based generalizations of the Cauchy equation.
In this spirit, Fischer and Musz\'{e}ly~\cite{FM} proposed and studied the norm-type functional equation
\begin{equation}\label{FM}
\|T(x+y)\| = \|T(x) + T(y)\| \qquad (x, y \in Q),
\end{equation}
where $T$ is a mapping from a semigroup $(Q, +)$ into a Banach space $(X, \|\cdot\|)$.
Equation~\eqref{FM}, now known as the Fischer--Musz\'{e}ly equation, captures the idea of norm-preserving additivity.
They showed that if $T : Q \to X$ satisfies \eqref{FM} and the target space $X$ is a Hilbert space, then $T$ must be additive.
In contrast, when $X$ is a non-strictly convex Banach space, they constructed examples of mappings $T$ satisfying \eqref{FM} which are not additive.
These counterexamples illustrate that strict convexity is a necessary condition for \eqref{FM} to imply additivity in general Banach spaces.

The relationship between the functional equations \eqref{Ca} and \eqref{FM} has been further investigated in later works.
Ger \cite{RG} proved that for a group $G$ and a Banach space $X$, every mapping $T : G \to X$ satisfying \eqref{FM} is additive if and only if $X$ is strictly convex.
In a related setting, Sch\"{o}pf \cite{PS} provided characterizations of mappings $T : \mathbb{R} \to X$ satisfying \eqref{FM} under continuity or differentiability assumptions.
A refinement was given by Maksa and Volkman \cite{MV}, who showed that any mapping $T$ from a group $G$ into a Hilbert space $X$ satisfying
\[
\|T(x+y)\| \geq \|T(x) + T(y)\| \qquad (x, y \in G)
\]
must be additive.

While Ger's result shows that the additivity of every mapping satisfying \eqref{FM} characterizes the strict convexity of the Banach space $X$, it naturally leads to the question:
\[
\textit{What if the mapping $T : G \to X$, defined on a group $(G, +)$, is assumed to be surjective?}
\]
This case was addressed by Tabor \cite{JT}, who showed that if $T$ is a surjective mapping from a group $G$ onto a Banach space $X$ and satisfies \eqref{FM}, then $T$ must be additive, even if $X$ is not strictly convex.
It should be emphasized that the group structure of $G$ is essential to Tabor's proof and cannot be omitted.
These results suggest that preserving the norm of the sum of two elements may also preserve the underlying algebraic structure of the Banach space.

As part of a parallel line of investigation, L.~Moln\'ar has carried out a series of significant studies on nonlinear transformations defined on the positive cones of \(C^*\)-algebras, that is, the sets of positive (semidefinite or definite) elements.
His work has focused on identifying structural properties preserved by such transformations, particularly those that maintain various operator means, including the arithmetic, geometric, and harmonic means~\cite{Mol1}, as well as power means~\cite{Mol2} and Kubo--Ando means~\cite{Mol3}.
These developments culminated in the joint work of Dong, Li, Moln\'ar, and Wong~\cite{DLMW}, where the authors systematically studied norm-preserving transformations associated with various operator means---especially the arithmetic, geometric, and harmonic means---and showed that such maps, under suitable assumptions, can be extended either to \(*\)-isomorphisms or to Jordan \(*\)-isomorphisms. 

These findings further raise the general question of how norm-related information may influence or reflect other mathematical structures, including underlying algebraic properties.
Against this backdrop, and motivated by his own contributions to the structure theory of positive cones, L.~Moln\'ar proposed the following natural problem:
Can norm-type functional equations such as
\begin{equation*}
\|T(x + y)\| = \|T(x) + T(y)\|
\end{equation*}
be meaningfully analyzed in the framework of positive cones?

To explore this question more concretely, we consider the Banach space \( C_0(L) \), consisting of all complex-valued continuous functions vanishing at infinity on a locally compact Hausdorff space \( L \), equipped with the supremum norm:
\[
\|f\| = \sup_{x \in L} |f(x)| \quad (f \in C_0(L)).
\]
Throughout this paper, we shall use the supremum norm for all continuous functions.
Our attention will be focused on the positive cone
\[
C_0^+(L) = \{ f \in C_0(L) \mid f \geq 0 \},
\]
which is closed under pointwise addition and thus forms a commutative semigroup with respect to addition.
It is worth noting that the Banach space \( C_0(L) \) is not strictly convex, which distinguishes it from the class of Banach spaces considered in earlier results, such as those involving Ger's theorem.

It was within this context that L.~Moln\'ar presented to the author a precise formulation of the problem.
His insight is that, although the equation~\eqref{FM} was initially studied in the general setting of Banach spaces, it can naturally be extended to the setting of positive cones in 
\(C^*\)-algebras in a meaningful and structurally revealing way. 
L.~Moln\'ar posed the following problem:

\begin{problem}[Moln\'{a}r]
Let $\pcz$ be the positive cone of a commutative $C^{*}$-algebra  
 $C_0(L_i)$ for $i=1,2$. 
If $T: \pcx\to \pcy$ is a surjective mapping satisfying
\begin{equation}\label{1-3}
\left\|T(f+g)\right\|=\left\|T(f)+T(g)\right\|\qquad (f,g\in \pcx), 
\end{equation}
then is the map $T$  additive?
\end{problem}

In this paper, we provide a solution to this Problem and establish the following main result.

\begin{thm}\label{thm1}
Let $\pcz$ be the positive cone of a commutative $C^{*}$-algebra $C_0^{+}(L_i)$ for $i=1,2$. 
If $T\colon \pcx \to \pcy$ is a surjective mapping satisfying
\begin{equation*}
\|T(f+g)\|=\|T(f)+T(g)\|\qquad (f,g\in \pcx), 
\end{equation*}
then 
$T$ is additive and positive homogeneous. 
\end{thm}

 As a consequence of Theorem~\ref{thm1}, 
 we also obtain the following cororally. 

\begin{cor}\label{cor1}
Let $A^{+}_{i}$ be the positive cone of a unital commutative $C^{*}$-algebra $A_i$ with the unit element $1_{A_{i}}$ for $i=1,2$. 
If  $T:A^{+}_{1}\to A^{+}_{2} $ a bijective mapping satisfying  
\[
\left\|T(a+b)\right\|=\left\|T(a)+T(b)\right\|\qquad (a,b\in A_1^{+}),   
\]
then  
there exists a homeomorphism $\tau:K_2\to K_1$ such that 
\[
\wh{T(a)}(\xi)=\wh{T(1_{A_{1}})}(\xi)\wh{a}(\tau(\xi))\qquad (a\in A_{1}^{+},\:\xi\in K_2), 
\]
where we denote by $K_i$ the maximal ideal space of $A_i$ and $\wh{a}$ the Gelfand transform of $a\in A_{i}$  for $i=1,2$.  
\end{cor}

\section{Preliminaries}\label{2}
Let  
$C_0(L)$ the commutative $C^{*}$-algebra
of all complex-valued continuous functions on a locally compact Hausdorff space $L$ vanishing at infinity 
equipped with the supremum norm 
 $\sn{f}=\sup_{x\in L}|f(x)|$ 
for $\nobreak{f\in C_0(L)}$.  
The positive cone, $C_0^{+}(L)$, is the set of all non-negative functions of $C_0(L)$, that is, 
\[
C_0^{+}(L)=\set{h\in C_0(L),\::\: h\geq 0}. 
\]
From now until section~\ref{sec4}, 
we 
assume that 
$T:\pcx\to \pcy$ is a surjective mapping satisfying  
 \begin{equation}\label{n-a}
 \left\|T\left(f+g\right)\right\|=\left\|T(f)+T(g)\right\|\qquad (f,g\in \pcx). 
 \end{equation}
 In addition, we denote by \( \mathbb{N} \) the set of all natural numbers,  
by \( \mathbb{R} \) the field of real numbers,  
and by \( \mathbb{C} \) the field of complex numbers.
 First, it is worth noting that the following can be derived from \eqref{n-a}. 

\begin{rem}\label{r1}
Let $f, g, h\in \pcx$. 
 We apply \eqref{n-a} to the map $T:\pcx\to \pcy$ repeatedly, and then, 
 it follows that 
\[
\left\|T(f+g)+T(h)\right\|=\left\|T\left((f+g)+h\right)\right\|=\left\|T\left(f+(g+h)\right)\right\|=\left\|T(f)+T(g+h)\right\|.  
\]
Hence, we obtain $\left\|T(f+g)+T(h)\right\|=\left\|T(f)+T(g+h)\right\|$ for all $f,g,h\in \pcx$. 
\end{rem}

Next, we prove that  $T(nf)$ attains its maximum at the same point in $L_2$ as $T(f)$ for $f\in \pcx$ and $n\in \N$. 
We prepare the following two lemmas 
in order to prove that the map $T$, which satisfies \eqref{n-a}, preserves scalar mutiplication 
for any $n\in \N$ with respect to the norm. 
Before proving them, we prepare some symbols as below.
\begin{nota}
For each $y_0\in L_2$, we define a subset $\pp{y_0}$ of $\pcy$ by 
\begin{equation*}
\pp{y_0}=\set{h\in \pcy\: : \: h(y_0)=\Vinf{h}}. 
\end{equation*}
For each $h\in \pcy$, we denote by
\[
M_{h}=\set{y\in L_2\: : \:h(y)=\Vinf{h} }. 
\]
Note that $M_{h}$ is a compact subset of $L_2$ if $\|h\|>0$. 
\end{nota}

\begin{lem}\label{b-1}
Let $f\in\pcx$.  If $y_0\notin M_{T(f)}$, 
then there exists $g\in \pcx$ such that  $T(g)\in\nobreak \pp{y_0}$ 
 and 
$
\Vinf{T(f)+T(g)}<\Vinf{T(f)}+\Vinf{T(g)}.
$
\end{lem}

\begin{proof}
Because $\{y_0\}\cap M_{T(f)}$ is empty and $M_{T(f)}$ is compact in $L_2$, 
we can choose $h\in \pcy$ such that 
$
h(y_0)=1=\Vinf{h}
$
 and 
 $
 h=0
 $ on $M_{T(f)}$ by Urysohn's lemma.  
 There exists $g\in \pcx$ such that $T(g)=h$, 
since $T:\nobreak\pcx\to \pcy$ is surjective. 
It follows from $T(g)=h$ with the choice of $h$ that $T(g)(y_0)=1=\sn{T(g)}$. 
The remaining thing we need to prove is 
 that $T(f)(y)+T(g)(y)<\Vinf{T(f)}+\Vinf{T(g)}$ for all $y\in L_2$. 

If $y\in M_{T(f)}$, then 
$$
T(f)(y)+T(g)(y)=T(f)(y)+h(y)=T(f)(y)+0<\Vinf{T(f)}+\Vinf{T(g)}, 
$$
and thus, $T(f)(y)+T(g)(y)<\sn{T(f)}+\sn{T(g)}$. 
Next, we assume that $y\in L_2\setminus M_{T(f)}$.  
 It follows from $T(f)(y)<\Vinf{T(f)}$ that  $T(f)(y)+T(g)(y)<\Vinf{T(f)}+\Vinf{T(g)}$. 
Therefore, we conclude that 
$
T(f)(y)+T(g)(y)<\sn{T(f)}+\sn{T(g)}
$ 
for all $y\in L_2$. 
\end{proof}

\begin{lem}\label{2-3-4}
Let \( h_1, h_2 \in \pcy \). Then the following statements are equivalent\:: 
\begin{enumerate}[\rm{(i)}]
\item[\rm{(i)}] \( \|h_1\| + \|h_2\| = \|h_1 + h_2\| \);
\item[\rm{(ii)}] \( M_{h_1} \cap M_{h_2} \text{ is non-empty}\).
\end{enumerate}  
\end{lem}

\begin{proof}
Assume that \rm{(i)} holds. Then there exists \( y_0 \in L_2 \) such that
\[
\|h_1 + h_2\| = h_1(y_0) + h_2(y_0).
\]
By assumption \rm{(i)}, we have
\[
\|h_1\| + \|h_2\| = h_1(y_0) + h_2(y_0).
\]
Since \( h_2(y_0) \leq \|h_2\| \), it follows that
\[
\|h_1\| + \|h_2\| = h_1(y_0) + h_2(y_0) \leq h_1(y_0) + \|h_2\| \leq \|h_1\| + \|h_2\|,
\]
and hence all inequalities must be equalities. This implies that \( h_1(y_0) = \|h_1\| \) and \( h_2(y_0) = \|h_2\| \), that is, \( y_0 \in M_{h_1} \cap M_{h_2} \). Therefore, \rm{(ii)} holds.

Conversely, assume that \rm{(ii)} holds. Then there exists \( y_0 \in M_{h_1} \cap M_{h_2} \) such that
\[
h_1(y_0) = \|h_1\| \quad \text{and} \quad h_2(y_0) = \|h_2\|.
\]
Applying the triangle inequality yields
\[
\|h_1 + h_2\| \leq \|h_1\| + \|h_2\| = h_1(y_0) + h_2(y_0) \leq \|h_1 + h_2\|.
\]
Thus, all inequalities are equalities, and we conclude that
\[
\|h_1\| + \|h_2\| = \|h_1 + h_2\|,
\]
which proves \rm{(i)}. The proof is complete.
\end{proof}

The next lemma plays an important role in proving that \( M_{T(nf)} \subset M_{T(f)} \) for \( n \in \mathbb{N} \) and \( f \in \pcx \).

\begin{lem}\label{3-1}
Let \( f, g \in \pcx \). If \( M_{T(f)} \cap M_{T(g)} \) is non-empty, then 
\[
M_{T(f+g)} \subset M_{T(f)}.
\]
\end{lem}

\begin{proof}
Suppose, for contradiction, that there exists \( y_0 \in M_{T(f+g)} \setminus M_{T(f)} \). 
Since \( y_0 \notin M_{T(f)} \), we can choose \( g_0 \in \pcx \) such that 
\( T(g_0) \in \pp{y_0} \) and 
\begin{equation}\label{2-3-1}
\|T(f) + T(g_0)\| < \|T(f)\| + \|T(g_0)\| 
\end{equation}
by Lemma~\ref{b-1}.
On the other hand, since \( y_0 \in M_{T(f+g)} \cap M_{T(g_0)} \), 
it follows from Lemma~\ref{2-3-4} that
\begin{equation}\label{2-3-2}
\|T(f+g) + T(g_0)\| = \|T(f+g)\| + \|T(g_0)\|.
\end{equation}
Given that \( M_{T(f)} \cap M_{T(g)}\) is non-empty, we deduce from 
Lemma~\ref{2-3-4} with the equation~\eqref{n-a} that
\[
\|T(f+g)\| = \|T(f) + T(g)\| = \|T(f)\| + \|T(g)\|.
\]
Combining the last equality with \eqref{2-3-2}, we obtain
\begin{equation*}
\|T(f)\| + \|T(g)\| + \|T(g_0)\| = \|T(f+g) + T(g_0)\|.
\end{equation*}
We deduce from Remark~\ref{r1} that
\begin{equation*}
\sn{T(f+g)+T(g_0)}=\sn{T(g+f)+T(g_0)}=\sn{T(g)+T(f+g_0)}. 
\end{equation*}
Hence, we obtain 
\begin{equation}\label{k-1}
\|T(f)\| + \|T(g)\| + \|T(g_0)\| = \|T(g) + T(f + g_0)\|.
\end{equation}
Note that $\sn{T(f+g_0)}=\sn{T(f)+T(g_0)}$ by \eqref{n-a}. 
Using inequality~\eqref{2-3-1}, we have
\[
\|T(g) + T(f + g_0)\| \leq \|T(g)\| + \|T(f + g_0)\| < \|T(g)\| + \|T(f)\| + \|T(g_0)\|.
\]
This contradicts the equality~\eqref{k-1}. 
Therefore, 
we must have $M_{T(f+g)}\subset M_{T(f)}$. 
\end{proof}

\begin{lem}\label{3-2}
For each $n\in \mathbb{N}$ and $f\in\pcx$, the inclusion 
$
M_{T(nf)}\subset M_{T(f)}
$
 holds. 
\end{lem}

\begin{proof}
Let \( f \in \pcx \) and fix it. 
We prove by induction on \( n \in \mathbb{N} \) that \( M_{T(nf)} \subset M_{T(f)} \) holds for all \( n \in \mathbb{N} \).
For the case \( n = 1 \), the statement is trivial, 
 since \( M_{T(nf)} = M_{T(f)} \).
 
Assume, as the induction hypothesis, that \( M_{T(kf)} \subset M_{T(f)} \) holds for some \( k \in \mathbb{N} \). 
Since \( M_{T(kf)} \cap M_{T(f)} = M_{T(kf)} \), this intersection is non-empty. 
Therefore, by Lemma~\ref{3-1}, we have
\[
M_{T((k+1)f)} = M_{T(kf + f)} \subset M_{T(kf)} \subset M_{T(f)}.
\]
This completes the induction step. 
Hence, \( M_{T(nf)} \subset M_{T(f)} \) holds for all \( n \in \mathbb{N} \), as desired.
\end{proof}

We are now ready to prove that the map \( T \), which satisfies \eqref{n-a}, 
preserves scalar multiplication by natural numbers with respect to the norm.

\begin{lem}\label{3-3}
For every \( n \in \mathbb{N} \) and \( f \in \pcx \), we have
\[
\|T(nf)\| = n\|T(f)\|.
\]
\end{lem}

\begin{proof}
Let \( f \in \pcx \) and fix it. 
We prove by induction on \( n \in \mathbb{N} \) that \( \|T(nf)\| = n\|T(f)\| \) holds.
For the case \( n = 1 \), the statement is trivial.

Assume, as the induction hypothesis, that \( \|T(kf)\| = k\|T(f)\| \) holds for some \( k \in \mathbb{N} \). 
By the equation~\eqref{n-a}, we have
\[
\|T((k+1)f)\| = \|T(kf + f)\| = \|T(kf) + T(f)\|.
\]
Since \( M_{T(kf)} \subset M_{T(f)} \) by Lemma~\ref{3-2}, their intersection is non-empty. 
Thus, by Lemma~\ref{2-3-4}, we obtain
\[
\|T(kf) + T(f)\| = \|T(kf)\| + \|T(f)\|.
\]
Substituting the induction hypothesis into this equality yields
\[
\|T((k+1)f)\| = \|T(kf)\| + \|T(f)\| = k\|T(f)\| + \|T(f)\| = (k+1)\|T(f)\|.
\]
This completes the induction step. 
Hence, \( \|T(nf)\| = n\|T(f)\| \) holds for all \( n \in \mathbb{N} \).
\end{proof}

Next, 
 we crarify the relationship among $M_{T(f)}$, $M_{T(g)}$, and $M_{T(f+g)}$ 
 when $M_{T(f)}\cap M_{T(g)}$ is non-empty for $f,g\in \pcx$. 
To this end, we make good use of the following lemma.

\begin{lem}\label{2-3-5}
Let $y\in L_2$ and $K\subset L_2$ a compact subset. 
If $K\cap V$ is non-empty for any open neighborhood $V$ of $y$, 
then $y\in K$.  
\end{lem}

\begin{proof}
We suppose that $y\notin K$, and then, 
$y\in L_2\setminus K$. 
Since $L_2$ is a locally compact Hausdorff space, 
there exists an open neighborhood $V_0$ of $y$ such that $y\in V_0\subset L_2\setminus K$. 
This shows that $V_0\cap K$ is empty for some open neighborhood $V_0$ of $y$.  
Therefore, 
we conclude that Lemma~\ref{2-3-5} holds by the contrapositive. 
\end{proof}

 The following lemma establishes that for any 
 $f\in \pcx$ and $n\in \N$ with $n\geq 2$,  the function $T(nf)$ attains its maximum at the same point as $T(f)$. 
 We denote by $\zlb$ the zero element in $\pcy$, 
  and 
  this notation will be consistently used throughout the rest of this paper.

\begin{lem}\label{2-3-7}
For each \( f \in \pcx \) and each \( n \in \mathbb{N} \) with \( n \geq 2 \), we have
\[
M_{T(nf)} = M_{T(f)}.
\]
\end{lem}

\begin{proof}
Let \( f \in \pcx \) and fix it. 
We distinguish two cases.

\textbf{Case 1.} Suppose that \( T(f) = \zlb \).

By Lemma~\ref{3-3}, we have
\[
\|T(nf)\| = n \|T(f)\| = 0
\]
for all \( n \in \mathbb{N} \) with \( n \geq 2 \). 
Therefore, we conclude that
$
T(f) = \zlb = T(nf),
$
which implies that 
\[
M_{T(nf)} = L_2 = M_{T(f)}.
\]

\textbf{Case 2.} Suppose that \( T(f) \neq \zlb \).

In this case, we first show that \( M_{T((m+1)f)} = M_{T(mf)} \) for all \( m \in \mathbb{N} \).
By Lemma~\ref{3-2}, we have \( M_{T(mf)} \subset M_{T(f)} \), 
and hence,  it follows that \( M_{T(mf)} \cap M_{T(f)} = M_{T(mf)} \) is non-empty.  
Applying Lemma~\ref{3-1} to the pair \( (T(mf), T(f))\), 
we obtain
\[
M_{T((m+1)f)} = M_{T(mf + f)} \subset M_{T(mf)}.
\]

To prove the reverse inclusion, take an arbitrary point \( y_0 \in M_{T(mf)} \) and fix it. 
Let \( V \) be any open neighborhood of \( y_0 \).  
By Urysohn's lemma, there exists \( h_V \in \pcy \) such that \( h_V(y_0) = 1 = \|h_V\| \) and \( h_V(y) = 0 \) for all \( y \notin V \).  
Since \( T \colon \pcx \to \pcy \) is surjective, there exists \( g_V \in \pcx \) such that \( T(g_V) = h_V \).  
It follows from  Remark~\ref{r1} that 
\begin{equation}\label{2-3-7-2}
\|T((m+1)f) + T(g_V)\| = \|T(f) + T(mf + g_V)\|.
\end{equation}
Since \( y_0 \in M_{T(mf)} \cap M_{T(g_V)} \), 
Lemma~\ref{3-1} shows that  \( M_{T(mf + g_V)} \subset M_{T(mf)} \),  
and Lemma~\ref{3-2} implies that  \( M_{T(mf)} \subset M_{T(f)} \).  
It follows  that \( M_{T(mf + g_V)} \subset M_{T(f)} \).  
Therefore, we conclude that 
\( M_{T(mf + g_V)} \cap M_{T(f)} \) is non-empty.  
By Lemma~\ref{2-3-4}, we obtain
\begin{equation}\label{2-3-7-3}
\|T(f) + T(mf + g_V)\| = \|T(f)\| + \|T(mf + g_V)\|.
\end{equation}
Since \( y_0 \in M_{T(mf)} \cap M_{T(g_V)} \), Lemma~\ref{2-3-4} and the  equation~\eqref{n-a} yield
\[
\|T(mf + g_V)\| = \|T(mf) + T(g_V)\| = \|T(mf)\| + \|T(g_V)\|.
\]
Having in mind that  \( \|T(mf)\| = m \|T(f)\| \) by Lemma~\ref{3-3},  
it follows that 
\[
\|T(mf + g_V)\| = m \|T(f)\| + \|T(g_V)\|.
\]
Combining this equality  with \eqref{2-3-7-2} and \eqref{2-3-7-3}, we get
\[
\|T((m+1)f) + T(g_V)\| = (m+1)\|T(f)\| + \|T(g_V)\|.
\]
Applying Lemma~\ref{3-3} to \(T\left((m+1)f\right)\), 
we obtain  \( \|T((m+1)f)\| = (m+1)\|T(f)\| \), and hence, it follows that 
\[
\|T((m+1)f) + T(g_V)\| = \|T((m+1)f)\| + \|T(g_V)\|.
\]
We deduce from Lemma~\ref{2-3-4} that \( M_{T((m+1)f)} \cap M_{T(g_V)}\) is non-empty.  
Since \( T(g_V)(y) = h_V(y) = 0 \) for all \( y \notin V \), we have \( M_{T(g_V)} \subset V \).  
Therefore, we observe that \( M_{T((m+1)f)} \cap V\) is non-empty.  
We deduce from Lemma~\ref{3-3} with $T(f)\neq \zlb$ that 
\[
\|T((m+1)f)\|=(m+1)\|T(f)\|>0. 
\]
This implies $M_{T((m+1)f)}$ is compact in $L_2$. 
As \( V \) is an arbitrary open neighborhood of $y_0$, 
Lemma~\ref{2-3-5} implies \( y_0 \in M_{T((m+1)f)} \).  
Since \( y_0 \in M_{T(mf)} \) is arbitrarily chosen, 
we conclude \( M_{T(mf)} \subset M_{T((m+1)f)} \), and hence, 
we obtain 
\[
M_{T((m+1)f)} = M_{T(mf)}.
\]
Finally, for each \( n \geq 2 \), applying this equality repeatedly yields \( M_{T(nf)} = M_{T(f)} \).  
\end{proof}

From the above lemma, 
it follows that the map $T:\pcx\to \pcy$ preserves scalar multiplication by natural numbers at the point $y\in M_{T(f)}$.

\begin{lem}\label{2-3-8}
For each $f\in\pcx$ and $n\in \N$ with $n\geq 2$, 
the following identity holds:
\[
T(nf)(y_0)=nT(f)(y_0)\qquad (y_0\in M_{T(f)}). 
\]
\end{lem}

\begin{proof}
Let $f\in \pcx$ and $n\in \N$ with $n\geq 2$. 
Choose $y_0\in M_{T(f)}$ arbitrarily and fix it. 
Applying  Lemma~\ref{2-3-7} to $T(f)$, 
we obtain $M_{T(nf)}=M_{T(f)}$, and thus, $y_0\in M_{T(nf)}\cap M_{T(f)}$. 
It follows from Lemma~\ref{3-3} that 
\[
T(nf)(y_0)=\sn{T(nf)}=n\sn{T(f)}=nT(f)(y_0). 
\]
Since $y_0\in M_{T(f)}$ is arbitrarily chosen, we see that Lemma~\ref{2-3-8} holds. 
\end{proof}

Now, 
we establish the equality $M_{T(f+g)}=M_{T(f)}\cap M_{T(g)}$ for functions 
$f,g\in \pcx$, assuming that $M_{T(f)}\cap M_{T(g)}$ is non-empty, as stated in the following lemma.

\begin{lem}\label{2-3-10}
Let $f, g\in \pcx$. If $M_{T(f)}\cap M_{T(g)}$ is non-empty, 
then 
\[
M_{T(f+g)}= M_{T(f)}\cap M_{T(g)}. 
\]
\end{lem}

\begin{proof}
Let \( f, g \in \pcx \) and fix them. 
We distinguish two cases.

\textbf{Case 1.} Suppose that \( T(f+g) = \zlb \).

It follows from \eqref{n-a} that 
\[
\|T(f)+T(g)\|=\|T(f+g)\|=0.  
\]
This implies that $T(f)=\zlb$ and $T(g)=\zlb$, 
since $T(f), T(g)\in \pcy$.  
Therefore, we obtain 
\[
M_{T(f+g)}=L_2=M_{T(f)}\cap M_{T(g)}. 
\]

\textbf{Case 2.} Suppose that \( T(f+g) \neq \zlb \). 

Since \( M_{T(f)} \cap M_{T(g)}\) is non-empty, 
Lemma~\ref{3-1} yields 
$
M_{T(f+g)} \subset M_{T(f)}.
$
By the commutativity \( T(f+g) = T(g+f) \), 
we may apply Lemma~\ref{3-1} again to obtain 
$
M_{T(f+g)} \subset M_{T(g)}.
$
Hence, we have
\[
M_{T(f+g)} \subset M_{T(f)} \cap M_{T(g)}.
\]

To prove the reverse inclusion, 
we choose  \( y_0 \in M_{T(f)} \cap M_{T(g)} \)  arbitrarily and fix it. 
Let \( V \) be any open neighborhood of \( y_0 \).  
By Urysohn's lemma, there exists \( h \in \pp{y_0} \) such that \( \|h\| = 1 \), \( h(y_0) = 1 \), and \( h(y) = 0 \) for all \( y \in L_2 \setminus V \).  
Set \( h_V = h \cdot T(g) \in \pcy \).  
Since \( T \colon \pcx \to \pcy \) is surjective, there exists \( g_V \in \pcx \) such that \( T(g_V) = h_V \).  
It follows from $T(g_V)=h\cdot T(g)$ that  
\begin{equation}\label{2-3-10-1}
M_{T(g_V)} = M_h \cap M_{T(g)}.
\end{equation}
Since $h\in \pp{y_0}$ and $y_0\in M_{T(f)}\cap M_{T(g)}$,  
we have \( y_0 \in M_h \cap M_{T(g)} \), that is,  \( y_0 \in M_{T(g_V)} \).  
Hence, 
we conclude that $y_0\in M_{T(f)}\cap M_{T(g_V)}$.  
Lemma~\ref{3-1} and \eqref{2-3-10-1} give
\[
M_{T(f + g_V)} \subset M_{T(g_V)} \subset M_{T(g)},
\]
which implies \( M_{T(f + g_V)} \cap M_{T(g)}\) is non-empty. 
Combining Lemma~\ref{2-3-4} and Remark~\ref{r1}, we obtain
\begin{equation}\label{2-3-10-3}
\|T(f + g) + T(g_V)\| = \|T(g) + T(f + g_V)\| = \|T(g)\| + \|T(f + g_V)\|.
\end{equation}
On the other hand, since \( y_0 \in M_{T(f)} \cap M_{T(g_V)} \), 
we deduce from Lemma~\ref{2-3-4} with \eqref{n-a} that 
\[
\|T(f + g_V)\| = \|T(f)\| + \|T(g_V)\|.
\]
Substituting into \eqref{2-3-10-3}, we get
\begin{equation}\label{2-3-10-4}
\|T(f + g) + T(g_V)\| = \|T(f)\| + \|T(g_V)\| + \|T(g)\|.
\end{equation}
Since \( M_{T(f)} \cap M_{T(g)}\) is non-empty, 
Lemma~\ref{2-3-4} ensures that
\[
\|T(f) + T(g)\| = \|T(f)\| + \|T(g)\|.
\]
Combining this equality with \eqref{n-a}, we have \( \|T(f + g)\| = \|T(f)\| + \|T(g)\| \).  
Thus, it follows from \eqref{2-3-10-4} that 
\[
\|T(f + g) + T(g_V)\| = \|T(f + g)\| + \|T(g_V)\|.
\]
Applying Lemma~\ref{2-3-4} again to \( T(f + g) \) and \( T(g_V) \), we have
$
M_{T(f + g)} \cap M_{T(g_V)}
$
is non-empty. 
Since \( h(y) = 0 \) for all \( y \notin V \), we have \( M_h \subset V \).  
We deduce from \eqref{2-3-10-1} that 
\[
M_{T(g_V)} = M_h \cap M_{T(g)} \subset V.
\]
This shows that 
$
M_{T(f + g)} \cap V
$
is non-empty. 
Since \( T(f+g) \neq \zlb \), the set \( M_{T(f+g)} \) 
is a non-empty compact subset of \( L_2 \).
Lemma~\ref{2-3-5} implies that \( y_0 \in M_{T(f + g)} \), 
as \( V \) is an arbitrary open neighborhood of \( y_0 \).   
Since \( y_0 \in M_{T(f)} \cap M_{T(g)} \) was arbitrarily chosen, 
we conclude that
\[
M_{T(f)} \cap M_{T(g)} \subset M_{T(f + g)}.
\]
Thus, we obtain the desired equality.
\end{proof}

\section{Key Lemmas}

We retain the notation introduced in the previous section.
To prepare for the main result, which will be established in the next section,
we present several key lemmas that play a central role in our argument.

\begin{lem}\label{sub1}
Let \( u \in \pcy \) and \( y_0 \in L_2 \).  
If \( u(y_0) < \|u\| \), then there exists \( v \in \pp{y_0} \) with \( \|v\| > 0 \) such that  
\( u + v \in \pp{y_0} \) and \( \|u + v\| = \|u\| \).
\end{lem}

\begin{proof}
Set \( r = \|u\| - u(y_0) > 0 \).  
For each \( n \in \mathbb{N} \), define a  subset \( F_n \subset L_2 \) by
\[
F_n = \left\{ y \in L_2 \, \middle| \, 
r\left(1 - \dfrac{1}{2^n}\right) \leq \|u\| - u(y) \leq r\left(1 - \dfrac{1}{2^{n+1}}\right) \right\}.
\]
Since \( u \in \pcy \), the set \( F_n \) is a compact subset in $L_2$ for each \( n\in \N \).  
As \( y_0 \notin F_n \), 
we may apply Urysohn's lemma to obtain a function \( v_n \in \pp{y_0} \) satisfying
\begin{equation}\label{sub1-2}
\|v_n\| = 1 \quad \text{and} \quad v_n(y) = 0 \quad \text{for all } y \in F_n.
\end{equation}
Define 
\(
v = r \sum_{n=1}^\infty v_n/2^n.
\)
Because the series converges absolutely with respect to the norm, 
we obtain \( v \in \pcy \).

We claim that this function \( v \) satisfies the desired properties.  
By definition, we have
\[
v(y_0) = r \sum_{n=1}^\infty \frac{v_n(y_0)}{2^n} = r, \quad \text{and} \quad \|v\| \leq r.
\]
This shows that \( v \in \pp{y_0} \) and \( \|v\| = r > 0 \).

Let \( y \in L_2 \) be arbitrary.  
First, suppose \( y \notin \bigcup_{n=1}^\infty F_n \).  
Then we obtain 
\[ 
\|u\| - u(y) = r = \|u\| - u(y_0).
\]   
This implies that  \( u(y) = u(y_0) \), and hence,   
 we conclude that 
\[
u(y) + v(y) \leq u(y_0) + r = \|u\|.
\]
Next, assume \( y \in F_m \) for some \( m \in \mathbb{N} \).  
Then by the definition of \( F_m \), we have
\[
r\left(1-\frac{1}{2^m}\right)\leq \|u\|-u(y). 
\]
We deduce from $r=\|u\|-u(y_0)$ that 
\begin{equation}\label{sub1-1}
u(y) \leq u(y_0) + \frac{r}{2^m}.
\end{equation}
Since \( v_m(y) = 0 \) by \eqref{sub1-2}, 
it follows that 
\[
v(y) = r \sum_{n \neq m} \frac{v_n(y)}{2^n} \leq r\left(1 - \frac{1}{2^m}\right) = r - \frac{r}{2^m}.
\]
Combining this inequalities with \eqref{sub1-1}, we get
\[
u(y) + v(y) \leq u(y_0) + r = \|u\|.
\]
Since \( y \in L_2 \) is arbitrary, 
it follows that  \( \|u + v\| \leq \|u\| \).  
On the other hand, 
we derive from $v(y_0)=\|v\|=r$ that   
\[ 
u(y_0) + v(y_0) = u(y_0) + r = \|u\|.  
\]
This shows that $\|u+v\|\geq u(y_0)+v(y_0)=\|u\|$. 
Therefore, we observe that 
 \( \|u + v\| = \|u\| \), and hence, \( u + v \in \pp{y_0} \).
\end{proof}

\begin{lem}\label{sub2}
Let \( u, v \in \pcy \). 
Then for each \( y_0 \in L_2 \), there exists \( h \in \pp{y_0} \) such that 
\[
u + h, \; v + h \in \pp{y_0}.
\]
\end{lem}

\begin{proof}
If both \( u \) and \( v \) belong to \( \pp{y_0} \), then setting \( h = \zlb \in \pcy \) clearly yields \( u + h, v + h \in \pp{y_0} \).

We next consider the case where \( u \notin \pp{y_0} \) and \( v \in \pp{y_0} \). 
Since \( u(y_0) < \|u\| \), Lemma~\ref{sub1} guarantees the existence of \( h \in \pp{y_0} \) such that \( u + h \in \pp{y_0} \) and \( \|u + h\| = \|u\| \).  
Because both \( v \in \pp{y_0} \) and \( h \in \pp{y_0} \), 
we derive from the triangle inequality that 
\[
\|v\| + \|h\|=v(y_0) + h(y_0)\leq \|v + h\| \leq \|v\| + \|h\|,   
\]
and thus, 
we obtain 
$v(y_0) + h(y_0)= \|v + h\|$. 
This shows that 
 \( v + h \in \pp{y_0} \).  
The same argument applies when \( u \in \pp{y_0} \) and \( v \notin \pp{y_0} \).

Finally, consider the case where \( u, v \notin \pp{y_0} \).  
By Lemma~\ref{sub1}, there exist \( h_u, h_v \in \pp{y_0} \) such that 
\[
u + h_u \in \pp{y_0}, \quad \|u + h_u\| = \|u\|, \quad v + h_v \in \pp{y_0}, \quad \|v + h_v\| = \|v\|.
\]
Define \( h = h_u + h_v \).  
Since both \( h_u, h_v \in \pp{y_0} \), their sum \( h \in \pp{y_0} \) as well.  
We deduce from $u+h_u, h_v\in \pp{y_0}$ that 
\[
\|u+h_u\|+\|h_v\|=\left(u(y_0)+h_u(y_0)\right)+h_v(y_0)\leq \|u+(h_u+h_v)\|\leq \|u+h_u\|+\|h_v\|, 
\]
and hence, $u(y_0)+h(y_0)=\|u+h\|$. 
This shows that \( u + h \in \pp{y_0} \).  
The same reasoning shows that \( v + h \in \pp{y_0} \) as well.  
This completes the proof.
\end{proof}

We now demonstrate that the map $T:\pcx\to\pcy$, which satisfies \eqref{n-a}, is additive under certain conditions. 
The next result asserts that if the functions $T(f)$ and $T(g)$ for $f,g\in \pcx$ attain their maximum at the same point, then 
the map $T$ is additive at that point.

\begin{lem}\label{e-1}
Let \( f, g \in \pcx \). If \( T(f), T(g) \in \pp{y_0} \) for some \( y_0 \in L_2 \),  
then \( T(f + g) \in \pp{y_0} \) and  
\[
T(f + g)(y_0) = T(f)(y_0) + T(g)(y_0).
\]
\end{lem}

\begin{proof}
Since \( T(f), T(g) \in \pp{y_0} \), it follows that \( y_0 \in M_{T(f)} \cap M_{T(g)} \). 
Hence, 
the intersection $M_{T(f)} \cap M_{T(g)}$ 
 is non-empty.  
We deduce from  Lemma~\ref{2-3-10} that 
\( y_0 \in M_{T(f+g)} \).  
This shows that  \( T(f + g) \in \pp{y_0} \).  
Combining  Lemma~\ref{2-3-4}  with the equation~\eqref{n-a}, 
it follows from  
 the assumption \( y_0 \in M_{T(f)} \cap M_{T(g)} \) that 
\[
\|T(f + g)\| = \|T(f)\| + \|T(g)\|.
\]
Since \( T(f), T(g), T(f + g) \in \pp{y_0} \), we obtain 
\[
T(f + g)(y_0) = \|T(f + g)\| = \|T(f)\| + \|T(g)\| = T(f)(y_0) + T(g)(y_0),
\]  
as desired.
\end{proof}

The following lemma states that if the functions $T(f)+T(g)$ and $T(g)$ attain their maximum at some point $y_0\in L_2$, 
the map $T$ is additive at  $y_0$.

\begin{lem}\label{m-1}
Let \( y_0 \in L_2 \) and suppose that \( T(f) \in \pcy \). 
If there exists \( T(g) \in \pp{y_0} \) such that \( T(f) + T(g) \in \pp{y_0} \), 
then \( T(f+g) \in \pp{y_0} \) and
\[
T(f+g)(y_0) = T(f)(y_0) + T(g)(y_0).
\]
\end{lem}

\begin{proof}
We first prove that \( T(f+g) \in \pp{y_0} \). Suppose, to the contrary, that
\[
T(f+g)(y_0) < \|T(f+g)\|.
\]
Then, by Lemma~\ref{sub1}, there exists \( h_0 \in \pp{y_0} \) with \( \|h_0\| > 0 \) such that
\begin{equation}\label{nn-1}
\|T(f+g) + h_0\| = \|T(f+g)\|.
\end{equation}
On the other hand, since \( T(f) + T(g) \in \pp{y_0} \) and \( T \) satisfies \eqref{n-a}, we have
\[
T(f)(y_0) + T(g)(y_0) = \|T(f) + T(g)\| = \|T(f+g)\|.
\]
Combining the equality \eqref{nn-1} with the previous equalities, 
we obtain
\begin{equation}\label{m-1-2}
T(f)(y_0) + T(g)(y_0) = \|T(f+g) + h_0\|.
\end{equation}
Since \( T \) is surjective, there exists \( g_0 \in \pcx \) such that \( T(g_0) = h_0 \). 
By Remark~\ref{r1}, it follows that
\[
\|T(f+g) + T(g_0)\| = \|T(f) + T(g + g_0)\|.
\]
Substituting into \eqref{m-1-2}, we obtain
\begin{equation}\label{m-1-3}
T(f)(y_0) + T(g)(y_0) = \|T(f) + T(g + g_0)\| \geq T(f)(y_0) + T(g + g_0)(y_0).
\end{equation}
Since both \( T(g) \) and \( T(g_0) = h_0 \) belong to \( \pp{y_0} \), we have \( y_0 \in M_{T(g)} \cap M_{T(g_0)} \). 
It follows from Lemma~\ref{e-1} that
\[
T(g + g_0)(y_0) = T(g)(y_0) + T(g_0)(y_0).
\]
Entering this equality  into \eqref{m-1-3} yields
\[
T(f)(y_0) + T(g)(y_0) \geq T(f)(y_0) + T(g)(y_0) + T(g_0)(y_0),
\]
which implies \( T(g_0)(y_0) \leq 0 \). However, since \( T(g_0) = h_0 \in \pp{y_0} \), we must have \( \|h_0\| = h_0(y_0) = T(g_0)(y_0) \geq 0 \). 
We conclude that  \( T(g_0)(y_0) = 0 \), and hence \( \|h_0\| = 0 \), 
which contradicts  \( \|h_0\| > 0 \). 
Therefore, 
we must have $T(f+g)(y_0)=\|T(f+g)\|$. 
This shows that  \( T(f+g) \in \pp{y_0} \).

Finally, since \( T(f+g), T(f)+T(g) \in \pp{y_0} \), and \( T \) satisfies \eqref{n-a}, it follows that
\[
T(f+g)(y_0) = \|T(f+g)\| = \|T(f) + T(g)\| = T(f)(y_0) + T(g)(y_0).
\]
This completes the proof.
\end{proof}

\section{Proof of the Main Result}\label{sec4}

We are now in a position to prove the main result.

\begin{proof}[\bf Proof of Theorem~\ref{thm1}]
Let $f, g \in \pcx$, and fix an arbitrary point $y_0 \in L_2$. 
We first show that
\[
T(f+g)(y_0) \leq T(f)(y_0) + T(g)(y_0).
\]
By Lemma~\ref{sub1}, there exist functions $h_f, h_g \in \pp{y_0}$ such that 
$T(f) + h_f, T(g) + h_g \in \pp{y_0}$ and 
\[ 
\|T(f) + h_f\| = \|T(f)\|, \quad \|T(g) + h_g\| = \|T(g)\|.
\]
If $T(f) \in \pp{y_0}$ (resp.~$T(g) \in \pp{y_0}$), we may take $h_f = \zlb$ (resp.~$h_g = \zlb$).
Since $T$ is surjective, there exist $f_0, g_0 \in \pcx$ such that $T(f_0) = h_f$ and $T(g_0) = h_g$. 
Hence, we obtain 
\[ 
T(f) + T(f_0), \: T(g) + T(g_0) \in \pp{y_0},
\]
and
\begin{equation}\label{mr-1}
\|T(f) + T(f_0)\| = \|T(f)\|, \quad \|T(g) + T(g_0)\| = \|T(g)\|.
\end{equation}
Applying Lemma~\ref{m-1} to the pairs $(T(f), T(f_0))$ and $(T(g), T(g_0))$ respectively,  we obtain 
\[
T(f + f_0), \: T(g + g_0) \in \pp{y_0},
\]
and
\begin{align*}
T(f + f_0)(y_0) = T(f)(y_0) + T(f_0)(y_0), \:
T(g + g_0)(y_0) = T(g)(y_0) + T(g_0)(y_0).
\end{align*}
Adding these two equalities, we observe that 
\begin{equation}\label{mr-2}
T(f + f_0)(y_0) + T(g + g_0)(y_0) = T(f)(y_0) + T(f_0)(y_0) + T(g)(y_0) + T(g_0)(y_0).
\end{equation}
Since $T(f + f_0), T(g + g_0) \in \pp{y_0}$, Lemma~\ref{e-1} yields
\begin{align*}
T\big((f + f_0) + (g + g_0)\big)(y_0) &= T(f + f_0)(y_0) + T(g + g_0)(y_0), \\
\left\|T\big((f + f_0) + (g + g_0)\big)\right\| &= T\big((f + f_0) + (g + g_0)\big)(y_0).
\end{align*}
Combining the last two equalities with \eqref{mr-2}, we obtain
\begin{equation}\label{mr-3}
\left\|T\big((f + f_0) + (g + g_0)\big)\right\| = T(f)(y_0) + T(f_0)(y_0) + T(g)(y_0) + T(g_0)(y_0).
\end{equation}
We infer from \eqref{n-a} that 
\[ 
\|T((f + f_0) + (g + g_0))\| =\|T((f + g) + (f_0 + g_0))\| =\|T(f+g) + T(f_0 + g_0)\|. 
\]
This implies that 
\begin{equation}\label{mr-4}
T(f+g)(y_0) + T(f_0 + g_0)(y_0) \leq \|T((f + f_0) + (g + g_0))\|.
\end{equation}
We deduce from Lemma~\ref{e-1} with $T(f_0), T(g_0)\in \pp{y_0}$ that 
\[ 
T(f_0)(y_0) + T(g_0)(y_0) = T(f_0 + g_0)(y_0).
\]
Combining \eqref{mr-3} and \eqref{mr-4} with the last equality, we have
\begin{equation}\label{s-1}
T(f+g)(y_0) \leq T(f)(y_0) + T(g)(y_0).
\end{equation}

We next prove the reverse inequality. By Lemma~\ref{sub2} and surjectivity of $T$, 
there exists $h \in \pcx$ such that $T(h), T(f+g) + T(h), T(g) + T(h) \in \pp{y_0}$. 
Hence, we obtain 
\begin{equation}\label{mr5-2}
\|T(f+g) + T(h)\| = T(f+g)(y_0) + T(h)(y_0).
\end{equation}
It follows from Remark~\ref{r1} that 
\begin{equation}\label{mr-5}
T(f)(y_0) + T(g + h)(y_0) \leq \|T(f) + T(g + h)\| = \|T(f+g) + T(h)\|.
\end{equation}
Note that $T(h), T(g)+T(h)\in \pp{y_0}$ by the choice of $h\in \pcx$. 
We deduce from  Lemma~\ref{m-1} that 
\[ 
T(g + h)(y_0) = T(g)(y_0) + T(h)(y_0).
\]
Substituting into \eqref{mr-5}, we derive from \eqref{mr5-2} that 
\[
T(f)(y_0) + T(g)(y_0) + T(h)(y_0) \leq T(f+g)(y_0) + T(h)(y_0),
\]
which implies that 
\begin{equation}\label{s2}
T(f)(y_0) + T(g)(y_0) \leq T(f+g)(y_0).
\end{equation}
Combining both inequalities \eqref{s-1} and \eqref{s2}, we obtain $T(f+g)(y_0) = T(f)(y_0) + T(g)(y_0)$ for all $y_0 \in L_2$, and hence $T(f+g) = T(f) + T(g)$.

We now prove that $T$ is positive homogeneous. Fix $f \in \pcx$. We show $T(nf) = nT(f)$ for all $n \in \N$ by induction. The  case $n=1$ is trivial. 
Assume, as the induction hypothesis, that \( T(kf) = kT(f) \) holds for some \( k \in \mathbb{N} \). 
Since $T$ is additive by the above argument, we have 
\[ 
T((k+1)f) = T(kf + f) = T(kf) + T(f) = (k+1)T(f).
\]
This completes the induction step. 
Thus, $T(nf) = nT(f)$ for all $n \in \N$.
Let $r = n/m$ be a positive rational number. 
Then, we observe that 
\[ 
mT(rf) = T\left((mr)f\right) =  T(nf) = nT(f).   
\]
Hence, 
it follows that $T(rf) = rT(f)$ for all positive rational number $r$.  
For irrational $r > 0$, 
we choose two sequences of positive rational numbers 
$\{s_n\}^{\infty}_{n=1}$ and $\{t_n\}^{\infty}_{n=1}$ such that 
$s_n\leq r\leq t_n$ for all $n\in \N$ and $\lim_{n\to\infty}s_n=r=\lim_{n\to\infty}t_n$. 
It follows from  $s_n\leq r$ for $n\in \N$ that  
 $rf=(r-s_n)f+s_n f$ and $(r-s_n)f, s_n f\in \pcx$.  
 As $T$ is additive by the above argument,  
 we see that 
$$
T(s_n f)\leq T(s_n f)+T((r-s_n)f)=T(rf), 
$$ 
and thus, $T(s_n f)\leq T(rf)$. 
Applying the same argument to the pair $(T(rf), T(t_n f))$, 
we obtain $T(rf)\leq T(t_n f)$. 
Hence, we observe that 
$$
T(s_n f)\leq T(r f)\leq T(t_n f)
$$ 
for $n\in \N$.  
Since $T(s_n f)=s_nT(f)$ and $T(t_n f)=t_nT(f)$ for each $n\in \N$, 
it follows that  
 $$
 rT(f)=\lim_{n\to\infty}s_nT(f)\leq T(rf)\leq \lim_{n\to \infty}t_nT(f)=rT(f),  
 $$
 and thus, $T(rf)=rT(f)$. 
Therefore, we conclude that $T(rf)=rT(f)$ for all positive numbers $r$. 
 Since $f\in \pcx$ is arbitrarily chosen, 
 this completes the proof. 
\end{proof}

\section{Application of main result}\label{s-5} 
In this section, 
we prove Corollary~\ref{cor1} as an application of Theorem~\ref{thm1}.
Let $K_i$ be the maximal ideal space of the unital commutative $C^{*}$-algebra $A_i$ with the unit element $1_{A_i}$ and 
$C(K_i)$ the unital commutative $C^{*}$-algebra 
of all complex-valued  continuous functions on a compact Hausdorff space $K_i$ equipped with the supremum norm for $i=1,2$. 
We denote by 
 $\wh{a}\in C(K_i)$ the Gelfand transform of $a\in A_i$ defined by $\wh{a}(\eta)=\eta(a)$ for $\eta\in K_i$. 
By the Gelfand--Naimark theorem \cite[VIII, Theorem~2.1]{Conway}, 
the unital commutative $C^{*}$-algebras $A_i$ and $C(K_i)$ are isometrically *-isomorphic. 
Hence, 
we may regard $a\in A_i$ as a continuous function $\wh{a}$ on $K_i$ for each $a\in A_{i}$. 
In the rest of this section, 
we assume that 
 $T:A_1^{+}\to A_{2}^{+}$ is a bijective mapping satisfying 
\begin{equation}\label{nn1-0}
\left\|T\left(a+b\right)\right\|=\left\|T(a)+T(b)\right\|\qquad(a, b\in A_1^{+}). 
\end{equation}
Define $\wT:C^{+}(K_1)\to C^{+}(K_2)$ by 
\[
\wT(\wh{a})=\wh{T(a)}\qquad(\wh{a}\in C^{+}(K_1)). 
\]
It is easy to check that $\wT: C^{+}(K_1)\to C^{+}(K_2)$ is a well-defined bijective mapping satisfying 
\begin{equation}\label{nn1}
\left\|\wT(\wh{a}+\wh{b})\right\|=\left\|\wT(\wh{a})+\wT(\wh{b})\right\|\qquad(\wh{a}, \wh{b}\in C^{+}(K_1)). 
\end{equation}

Before proving Corollary~\ref{cor1}, we prepare several lemmas. 
The next lemma is essential for demonstrating that the map $\wT:C^{+}(K_1)\to C^{+}(K_2)$ preserves invertible elements.

\begin{lem}\label{nn2}
The map $\wT: C^{+}(K_1) \to C^{+}(K_2)$ is an order isomorphism. 
\end{lem}

\begin{proof}
Since $\wT: C^{+}(K_1) \to C^{+}(K_2)$ satisfies \eqref{nn1}, it follows from Theorem~\ref{thm1} that $\wT$ is additive and positive homogeneous. 
Let $\wh{a}, \wh{b} \in C^{+}(K_1)$ be such that $\wh{a} \leq \wh{b}$. 
Then $\wh{b} - \wh{a} \in C^{+}(K_1)$, and hence, we obtain 
\[
\wT(\wh{a}) \leq \wT(\wh{a}) + \wT(\wh{b} - \wh{a}) = \wT(\wh{b}),
\]
which shows that $\wT$ is order-preserving.

Furthermore, since $\wT$ is bijective, its inverse $\wT^{-1}: C^{+}(K_2) \to C^{+}(K_1)$ is also additive and positive homogeneous. 
By the same argument, we conclude that $\wT^{-1}$ is order-preserving as well. 
Therefore, $\wT$ is an order isomorphism.
\end{proof}

\begin{lem}\label{nn3}
Let \( C^{+}(K_i)^{-1} \) denote the set of all invertible elements in \( C^{+}(K_i) \) for \( i = 1, 2 \). Then the following equality holds:
\[
\wT(\ko) = \kt.
\]
\end{lem}

\begin{proof}
Let \( \wh{a} \in \ko \) and fix it. 
Because \( \wT: C^{+}(K_1) \to C^{+}(K_2) \) is surjective, there exists \( \wh{a_1} \in C^{+}(K_1) \) such that \( \wT(\wh{a_1}) = \wh{1_{A_2}} \). 
Now, \( \wh{a} \in \ko \) implies that there exists \( r > 0 \) such that
\[
r \wh{a_1}(\eta) \leq r \| \wh{a_1} \| \leq \wh{a}(\eta) \quad (\eta \in K_1),
\]
and hence, we obtain \( r \wh{a_1} \leq \wh{a} \). 
Given that \( \wT \) is an order isomorphism and positive homogeneous by  Lemma~\ref{nn2} and Theorem~\ref{thm1}, we have
\[
0 < r = r \wh{1_{A_2}}(\xi) = r \wT(\wh{a_1})(\xi) = \wT(r \wh{a_1})(\xi) \leq \wT(\wh{a})(\xi) \quad (\xi \in K_2),
\]
which shows that \( \wT(\wh{a})(\xi) > 0 \) for all \( \xi \in K_2 \). 
Thus, we observe that \( \wT(\wh{a}) \in \kt \). 
 Therefore, it follows that  \( \wT(\ko) \subset \kt \).

In order to  show the reverse inclusion, we apply the same reasoning to the inverse map \( \wT^{-1}: C^{+}(K_2) \to C^{+}(K_1) \), which is also an order isomorphism and positive homogeneous. It follows that \( \wT^{-1}(\kt) \subset \ko \), and hence \( \kt \subset \wT(\ko) \).
We conclude that \( \wT(\ko) = \kt \).
\end{proof}

We establish the following lemma as a preliminary step toward extending the map \( S: C^{+}(K_1) \to C^{+}(K_2) \) to the whole space \( C(K_1) \).  
In what follows, we will define the extended map
 \( S: C(K_1) \to C(K_2) \) in  proof of Cororally~\ref{cor1}.

\begin{lem}\label{dec}
Let $\widehat{a} \in C(K)$. Then there exist unique elements  
$\widehat{b_i}, \widehat{c_i} \in C^{+}(K)$ such that $\widehat{b_i}\widehat{c_i}=0$ 
for $i=1,2$ 
and 
\[
\widehat{a} = (\widehat{b_1} - \widehat{c_1}) + i(\widehat{b_2} - \widehat{c_2}).
\]
\end{lem}

\begin{proof}
Fix $\widehat{a} \in C(K)$, and denote by \( C_{\mathbb{R}}(K) \) the set of all self-adjoint elements in \( C(K) \), that is, the set of all real-valued continuous functions on \( K \).  
Define  
\[
\widehat{a_1} = \frac{\widehat{a} + \widehat{a^*}}{2} \quad \text{and} \quad \widehat{a_2} = \frac{\widehat{a} - \widehat{a^*}}{2i},
\]
where \( \widehat{a^*} \) denotes the adjoint of \( \widehat{a} \).  
Since \( \widehat{a_1}, \widehat{a_2} \in C_{\mathbb{R}}(K) \), it follows from \cite[Proposition 4.2.3]{KR} that there exist unique elements 
\( \widehat{b_i}, \widehat{c_i} \in C^{+}(K) \) such that $\widehat{b_i}\widehat{c_i}=0$ 
and 
\[
\widehat{a_i} = \widehat{b_i} - \widehat{c_i}
\]
 for  $i = 1, 2$. 
Combining these expressions, we obtain  
\[
\widehat{a} = (\widehat{b_1} - \widehat{c_1}) + i(\widehat{b_2} - \widehat{c_2}).
\]

To prove uniqueness, suppose that there exist elements \( \widehat{b_i'}, \widehat{c_i'} \in C^{+}(K) \) for \( i = 1, 2 \) such that  $\widehat{b_i'}\widehat{c_i'}=0$ and 
\[
\widehat{a} = (\widehat{b_1'} - \widehat{c_1'}) + i(\widehat{b_2'} - \widehat{c_2'}).
\]
By comparing the real part, we have  
\[
\widehat{b_1'} - \widehat{c_1'} = \frac{\widehat{a} + \widehat{a^*}}{2}   = \widehat{b_1} - \widehat{c_1},
\]
and since this lies in \( C_{\mathbb{R}}(K) \), the uniqueness in \cite[Proposition 4.2.3]{KR} yields \( \widehat{b_1'} = \widehat{b_1} \) and \( \widehat{c_1'} = \widehat{c_1} \).  
A similar argument applied to the imaginary part gives  
\( \widehat{b_2'} = \widehat{b_2} \) and \( \widehat{c_2'} = \widehat{c_2} \).  
This completes the proof.
\end{proof}

We are now in position to prove Corollary~\ref{cor1}.

\begin{proof}[\bf Proof of Corollary~\ref{cor1}]

We define $S:C^{+}(K_1)\to C^{+}(K_2)$ by $S(\wh{a})=\wT(\wh{1_{A_1}})^{-1}\wT(\wh{a})$ for $\wh{a}\in C^{+}(K_1)$. 
Since $\wT: C^{+}(K_1)\to C^{+}(K_2)$ is a bijective map which satisfies the condition~\eqref{nn1}, 
we deduce from Theorem~\ref{thm1} that $T$ is additive and positive homogeneous. 
Therefore, the map $S:C^{+}(K_1)\to C^{+}(K_2)$ is also a bijective additive and positive homogeneous map.  
Moreover, it follows from the definiton of the map $S$ with Lemmas~\ref{nn2} and \ref{nn3} that 
the map $S:C^{+}(K_1)\to C^{+}(K_2)$ is an order isomorphism which satisfies 
$S(\wh{1_{A_1}})=\wh{1_{A_2}}$.  
Choose $\wh{a}\in C^{+}(K_1)$ arbitrarily and fix it. 
By Lemma~\ref{dec}, 
there exist unique elements  
$\wh{b_i},\wh{c_i}\in C^{+}(K_1)$ such that 
$\widehat{b_i}\widehat{c_i}=0$ 
for $i=1,2$ and 
$\wh{a}=(\wh{b_1}-\wh{c_1})+i(\wh{b_2}-\wh{c_2})$. 
We define \( S_1(\wh{a})\in  C(K_2) \) by  
\begin{equation}\label{S1}
S_1(\widehat{a})=
(S(\widehat{b_1})-S(\wh{c_1})) +i(S(\widehat{b_2})-S(\wh{c_2})). 
\end{equation}
This definition is well-defined, since the decomposition 
$\widehat{a}= (\widehat{b_1} - \widehat{c_1})+i(\widehat{b_2} - \widehat{c_2})$ 
is unique by Lemma~\ref{dec}. 
Because $\wh{a}\in C(K_1)$ is arbitrarily chosen, 
we conclude that the map $S_1:C(K_1)\to C(K_2)$ is a well-defined map. 
We prove that the map $S_{1}$ is a linear order isomorphism. 
Because the map $S:C^{+}(K_1)\to C^{+}(K_2)$ is bijective, additive and 
positive homogeneous, 
it is straightforward to verify that $S_1: C(K_1)\to C(K_2)$ inherits these properties.

First, we prove that the map  \( S_1 \) satisfies \( S_1(\lambda \wh{a}) = \lambda S_1(\wh{a}) \) for all 
complex numbers \( \lambda \in \mathbb{C} \) and all \( \wh{a} \in C(K_1) \), that is, \( S_1 \) is homogeneous.
Take $\wh{a}\in C(K_1)$ arbitrarily. 
Then there exist unique elements $\wh{b_i}, \wh{c_i}\in C^{+}(K_1)$ for $i=1,2$ 
such that 
$\wh{a}=(\widehat{b_1} - \widehat{c_1}) + i(\widehat{b_2} - \widehat{c_2})$ 
 by Lemma~\ref{dec}. 
Because $-\widehat{a}=(\widehat{c_1} - \widehat{b_1})+i(\widehat{c_2}-\widehat{b_2})$, 
we deduce from the definition of $S_{1}$ that 
\[
S_{1}(-\wh{a})=(S(\widehat{c_1}) - S(\widehat{b_1}))+i(S(\widehat{c_2})-S(\widehat{b_2}))=-S_1(\wh{a}). 
\] 
This implies that $S_1(-\wh{a})=-S_1(\wh{a})$ for all $\wh{a}\in C(K_1)$. 
Since $S_1$ is positive homogeneous, 
we observe that $S_1(r \wh{a})=rS_1(\wh{a})$ for any $r\in \mathbb{R}$ and 
 $\wh{a}\in C(K_1)$. 
Having in mind that 
$i \wh{ a}=(\widehat{c_2} - \widehat{b_2}) + i(\widehat{b_1} - \widehat{c_1})$, 
we obtain 
\[
S_1(i\wh{a})=(S(\wh{c_2})-S(\wh{b_2}))+i(S(\wh{b_1})-S(\wh{c_1}))
           =i\left((S(\wh{b_1})-S(\wh{c_1}))+i(S(\wh{b_2})-S(\wh{c_2}))\right)
           =iS_1(\wh{a}). 
\]
This implies that 
$
S_{1}(\lambda \wh{a})=S_1((r+it)\wh{a})=rS_1(\wh{a})+itS_{1}(\wh{a})=\lambda S_1(\wh{a})
$
 for any  $\lambda=r+it\in \mathbb{C}$ with $r,t\in \mathbb{R}$ and $\wh{a}\in C(K_1)$. 
Hence, we conclude that $S_1$ is homogeneous. 

Next, we show that if $\wh{a_1}\leq \wh{a_2}$, then $S_1(\wh{a_1})\leq S_1(\wh{a_2})$. 
Assume that $\wh{a_1}\leq \wh{a_2}$ for some $\wh{a_1}, \wh{a_2}\in C(K_1)$. 
Then we note that $\wh{a_1}, \wh{a_2}\in C_{\mathbb{R}}(K_1)$. 
By Lemma~\ref{dec}, 
there exist  unique elemens $\wh{b_i}, \wh{c_i}\in C^{+}(K_1)$ such that 
$\wh{a_i}=\wh{b_i}-\wh{c_i}$ for $i=1,2$. 
It follows from $\wh{a_1}\leq \wh{a_2}$ that 
$\wh{b_1}-\wh{c_1}\leq \wh{b_2}-\wh{c_2}$, and hence, we obtain 
$\wh{b_1}+\wh{c_2}\leq \wh{b_2}+\wh{c_1}$. 
Since $S:C^{+}(K_1)\to C^{+}(K_2)$ is additive and an order isomorphism, 
it follows  that 
$S(\wh{b_1})+S(\wh{c_2})\leq S(\wh{b_2})+S(\wh{c_1})$. 
We deduce from $\wh{a_i}=\wh{b_i}-\wh{c_i}$ for $i=1,2$ that 
\[
S_1(\wh{a_1})= S(\wh{b_1})-S(\wh{c_1})\leq S(\wh{b_2})-S(\wh{c_2})=S_1(\wh{a_2}),  
\]
that is, $S_1(\wh{a_1})\leq S_1(\wh{a_2})$. 
Therefore, we conclude that $S_1$ is a bijective linear order homomorphism. 

Because $S:C^{+}(K_1)\to C^{+}(K_2)$ is bijective map, it follows from \eqref{S1} that 
\[
S_1^{-1}(\wh{a})
=(S^{-1}(\widehat{b_1})-S^{-1}(\wh{c_1})) +i(S^{-1}(\widehat{b_2})-S^{-1}(\wh{c_2}))
\]
for any $\wh{a}=(\widehat{b_1} - \widehat{c_1})+i(\widehat{b_2} - \widehat{c_2})\in C(K_2)$ with $\wh{b_i}, \wh{c_i}\in C^{+}(K_2)$ for $i=1,2$. 
Applying the same argument to $S_1^{-1}$, we observe that 
$S^{-1}$ is also a bijective linear order homomorphism. 
Hence, we conclude that 
 $S_1:C(K_1)\to C(K_2)$ is a bijective linear order isomorphism satisfying  
  $(S_1)_{|C^{+}(K_1)}=S$ and $S_{1}(\wh{1_{A_1}})=\wh{1_{A_2}}$, 
  where $(S_1)_{|C^{+}(K_1)}$ denotes the restriction of the map $S_1$ to $C^{+}(K_1)$. 
 We deduce from \cite[Corollary 5]{RK} that $S_1$ is a $C^{*}$-isomorphism. 
Hence, by \cite[Theorem~4.1.8]{KR} and the Banach--Stone Theorem~\cite[p172, Theorem~2.1]{Conway},  there exists a homeomorphism $\tau: K_2\to K_1$ such that 
\[
S_1(\wh{a})(\xi)=\wh{a}(\tau(\xi))\qquad(\wh{a}\in C(K_1), \xi\in K_2). 
\]
Having in mind that $(S_1)_{|C^{+}(K_1)}=S$, we obtain 
$S(\wh{a})(\xi)=\wh{a}(\tau(\xi))$, that is, 
\[
\wT(\wh{a})(\xi)=\wT(\wh{1_{A_1}})(\xi)\wh{a}(\tau(\xi))
\] for all $\wh{a}\in C^{+}(K_1)$ and $\xi\in K_2$. 
\end{proof}

\section{Future Work and Open Problem}
The main results of this paper address the problem under the assumption that the underlying $C^*$-algebras are commutative. 
After discussing this topic with L.~Moln\'{a}r, he kindly pointed out  that the corresponding problem for general (not necessarily commutative) $C^*$-algebras remains unsolved. 
He expressly encouraged me to include this problem in the present paper as an open problem, as it represents a natural and significant direction for future research.

\begin{oproblem}[Moln\'{a}r]
Let $A_i^{+}$ be the positive cone of a (not necessarily commutative) $C^{*}$-algebra $A_{i}$ for $i=1,2$. 
Suppose that $T:A_1^{+} \to A_{2}^{+}$ is a surjective mapping satisfying 
\[
\|T(f)+T(g)\| = \|T(f+g)\| \qquad (f,g \in A_{1}^{+}).
\]
Is the map $T$ necessarily additive in this general, possibly noncommutative setting?
\end{oproblem}

\begin{ack}
The author expresses his sincere gratitude to Lajos Moln\'{a}r for proposing the problem 
and 
for his invaluable guidance and insightful suggestions throughout the development of this manuscript. 
His thoughtful advice has greatly enhanced the quality of this work.

He also wishes to express his deep appreciation to Osamu Hatori 
for engaging in valuable discussions and for offering constructive advice that helped refine the results presented in this paper. 
His support has been instrumental in improving this research.

This work was supported by JSPS KAKENHI Grant Number JP 24K22845. 
\end{ack}

\section*{Declaration of generative AI and AI-assisted technologies in the writing process}

During the preparation of this work, the author used ChatGPT (developed by OpenAI) to improve the clarity, grammar, and overall readability of the manuscript, and to receive feedback on the logical structure of the Introduction. After using this tool, the author reviewed and edited the content as needed and takes full responsibility for the content of the publication.

\end{document}